\documentclass[DIV12,11pt]{scrartcl}

\usepackage{amsthm,amssymb,amsmath}
\usepackage[noblocks]{authblk}

\usepackage{amsmath}
\usepackage{amsopn}
\usepackage{amssymb}
\usepackage{latexsym}
\usepackage{amsfonts}
\usepackage{amsthm}
\usepackage[all]{xy}
\usepackage{tikz}
\usetikzlibrary{decorations.markings}
\tikzstyle{vertex}=[circle, draw, inner sep=0pt, minimum size=6pt]
\setkomafont{title}{\normalfont \LARGE \bfseries}
\setkomafont{section}{\normalfont \Large \bfseries \boldmath}
\setkomafont{subsection}{\normalfont \large \bfseries}
\setkomafont{subsubsection}{\normalfont \normalsize \bfseries}
\setkomafont{descriptionlabel}{\itshape}
\setkomafont{paragraph}{\normalfont \bfseries \boldmath}

\newtheorem{thm}{Theorem}[section]
\newtheorem{pro}[thm]{Proposition}
\newtheorem{lem}[thm]{Lemma}
\newtheorem{cla}[thm]{Claim}

\newtheorem{cor}[thm]{Corollary}

\theoremstyle{definition}
\newtheorem{obs}[thm]{Observation}

\newtheorem{exa}[thm]{Example}
\newtheorem{defn}[thm]{Definition}

\newtheorem{conj}[thm]{Conjecture}

\newcommand{\een}{\end{enumerate}}
\newcommand{\blem}{\begin{lem}}
\newcommand{\elem}{\end{lem}}
\newcommand{\bcl}{\begin{cla}}
\newcommand{\ecl}{\end{cla}}
\newcommand{\ethm}{\end{thm}}
\newcommand{\bpr}{\begin{pro}}
\newcommand{\epr}{\end{pro}}
\newcommand{\bco}{\begin{cor}}
\newcommand{\eco}{\end{cor}}
\newcommand{\bcon}{\begin{conj}}
\newcommand{\econ}{\end{conj}}
\newcommand{\bde}{\begin{defn}}
\newcommand{\ede}{\end{defn}}
\newcommand{\bex}{\begin{exa}}
\newcommand{\eexa}{\end{exa}}
\newcommand{\bobs}{\begin{obs}}
\newcommand{\eobs}{\end{obs}}
\newcommand{\bexe}{\begin{exe}}
\newcommand{\eexe}{\end{exe}}

\begin{document}
%\large
\title{Enumeration of balanced finite group valued functions on directed graphs}%\footnote{An extended abstract of this paper has been accepted to CTW 2013.}}
%\author{Yonah Cherniavsky \\
%EndAName
%Ariel University, Israel\\
%yonahch@ariel.ac.il \and Avraham Goldstein \\
%EndAName
%City University, New-York, USA\\
%avraham.goldstein.nyc@gmail.com \and Vadim E. Levit \\
%EndAName
%Ariel University, Israel\\
%levitv@ariel.ac.il}
\author[1]{Yonah Cherniavsky}
\author[2]{Avraham Goldstein}
\author[3]{Vadim E. Levit}
\author[4]{Robert Shwartz}

\affil[1]{Department of Computer Science and Mathematics, Ariel University, Israel; yonahch@ariel.ac.il}
\affil[2]{City University, New-York, USA; avraham.goldstein.nyc@gmail.com}
\affil[3]{Department of Computer Science and Mathematics, Ariel University, Israel; levitv@ariel.ac.il}
\affil[1]{Department of Computer Science and Mathematics, Ariel University, Israel; robertsh@ariel.ac.il}
\date{\vspace*{-2em}}
\maketitle

\begin{abstract}
\textbf{Abstract.} A group valued function on a graph is called balanced if the product of its values along any cycle is equal to the identity element of the group. We compute the number of balanced functions from edges and vertices of a directed graph to a finite non-Abelian group. %We study this group in two cases: when we are allowed to walk against the direction of an edge taking the opposite value of the function and when we are not allowed to walk against the direction.
\end{abstract}

\section{Introduction}
Let $G$ be a finite group with the group operation denoted by $\cdot$ and the identity element denoted by $1$.
Let $\Gamma$ be a graph. Roughly speaking, a $G$-valued function $f$ on vertices and/or edges of $\Gamma$ is called {\it balanced} if the product of its values along any cycle of $\Gamma$ is $1$. %%Our cycles are not permitted to have repeating edges.

The study of balanced functions can be conducted in three cases:
\begin{enumerate}
\item The graph $\Gamma$ is directed with the set of vertices $V$ and the set of directed edges $E$. When traveling between the vertices, we are allowed to travel with or against the direction of the edges. The value of a function $f$ on $\bar{e}$, which represents traveling the edge $e$ against its direction, is equal to $(f(e))^{-1}$. In this context, when the function is defined on edges only, the pair $(\Gamma,f)$ is called a network or a directed network. In this paper we shall call this {\it the flexible case}, meaning that the direction of an edge does not forbid us to walk against it. In the case when the group $G$ is Abelian (in particular $\mathbb Z$ or $\mathbb R$) the notion of balanced functions on edges for the flexible case, for functions taking values only on the edges, is introduced in the literature under different names. Thus, for example, in \cite{BHN} the set of such functions is exactly $Im(d)$, and in \cite{KY}, in somewhat different language, that set is referred to as the set of consistent graphs. In~\cite{T} such functions have been introduced under the name ``color-coboundaries". Also they appear in literature under the name ``tensions". They have been extensively studied, recent examples include~\cite{BeckEtAl},~\cite{BreuerDall},~\cite{Chen},~\cite{Woodall},~\cite{K}. In~\cite{BCNS} balanced $\mathbb C$-valued functions on edges appear in a certain connection with geometric representations of the Coxeter group associated with a graph. When the group $G$ is Abelian, the set of balanced functions forms a group in a natural way, the structure of this group was studied in~\cite{CGL2}. In a rather common terminology introduced by Zaslavsky,~\cite{Z}, a pair of a graph and such a function on edges of a graph is called a gain graph.

\item The graph $\Gamma$ is directed with the set of vertices $V$ and the set of directed edges $E$, but we are only allowed to travel with the direction of the edges. In this paper we shall call this {\it the rigid case}. When $f$ takes values only on the edges then in some literature, following Serre,~\cite{Serre}, the flexible case is described as a particular instance of the rigid case by introducing the set $\mathbb{E}$ as the new set of directed edges of $G$ (the cardinality of $\mathbb{E}$ is twice that of $E$), denoting by $\bar{e}\in \mathbb{E}$ the inverse of the directed edge $e\in {E}$ and requiring $f(\bar{e})=(f(e))^{-1}$,~\cite{BHN},~\cite{Serre}.

\item The graph $\Gamma$ is undirected. The value of a function $f$ on an edge $e$ does not depend on the direction of the travel on $e$. The case of balanced functions $f:E\rightarrow \mathbb{R}$ is studied in~\cite{OI}, where these functions are called ``cycle-vanishing edge valuations". For an Abelian group $A$, the case of balanced functions $f:V\bigcup E\rightarrow A$ is first introduced and studied in~\cite{NI}. The group structure of the groups of balanced functions on edges, and balanced functions on vertices and edges of an undirected graph with values in an Abelian group is studied in~\cite{CGL}.
\end{enumerate}

Notice that if we take functions with values in a non-Abelian group then the sets of the balanced function on edges and balanced functions on vertices and edges do not inherit any natural group structure. However, to find the number of these functions for directed graphs when the non-Abelian group is finite is possible, and we do it in the present paper.

%The subject of this paper is the group structure and the relations between groups of functions, associated with the notion of balance, on a directed graph. Namely, we study the group structures of the groups of balanced functions for the flexible and the rigid cases and the relations between these two cases.

%In this article we calculate the groups of balanced functions on edges, balanceable functions on vertices and balanced functions on vertices and edges of a directed graph with values in an Abelian group for both flexible and rigid cases.

In what follows we say that a directed graph is {\it weakly connected} if its underlying undirected graph is connected.

For the basics of Graph Theory we refer to \cite{Diestel}.
\section{The flexible case}
Let $\Gamma=(V,E)$ be a weakly connected directed graph, possibly with loops and multiple edges. Let $v,w\in V$ be two vertices connected by an edge $e$; $v$ is the origin of $e$ and $w$ is the endpoint of $e$. For $e\in E$ denote by $\bar{e}$ the same edge as $e$ but taken in the opposite direction. Thus $\bar{e}$ goes from $w$ to $v$. Let $\mathbb E=\left\{e,\bar{e}\,|\,e\in E\right\}$.
\bde\label{pathdef} A path $P$ from a vertex $x$ to a vertex $y$ is an alternating sequence $v_1$, $e_1$, $v_2$, $e_2$,...,$v_{n}$, $e_{n}$ of vertices from $V$ and different edges from $\mathbb E$ such that $v_1=x$ and each $e_j$, for $j=1,...,n-1$, goes from $v_j$ to $v_{j+1}$ and $e_n$ goes from $v_n$ to $y$. We permit the same edge $e$ to appear in a path twice - one time along and one time against its direction, since this is regarded as using two different edges from $\mathbb E$.
\ede
We require our graphs to be weakly connected. Namely, any two different vertices of our graph $\Gamma$ can be connected by a path.
\bde A path $P$ from a vertex $x$ to itself is called a cycle.
\ede
We permit the trivial cycle, which is the empty sequence containing no vertices and no edges.
%\bde A cycle is called simple if it contains every vertex at most one time. In other words $v_i\ne v_j$ if $i\ne j$.
%\ede

\bde The length of a cycle is the number of its edges.
\ede

\bde A function $f:\mathbb E\rightarrow G$ such that $f(\bar e)=(f(e))^{-1}$ is called {\it balanced} if the product $f(e_1)\cdot f(e_2)\cdots f(e_n)$ of the values of $f$ over all the edges of any cycle of $\Gamma$ is equal to $1$.
\ede
\bde The set of all the balanced functions $f:\mathbb E\rightarrow G$ is denoted by $\mathcal{FE}(\mathbb E,G)$. %$\mathcal{HF}(\mathbb E,A)$ is a subgroup of the Abelian group $A^{\mathbb E}$ of all the functions from $\mathbb E$ to $A$.
\ede
%\bde A function $g:V\rightarrow A$ is called {\it balanceable} if there exists some $f:\mathbb E\rightarrow A$ such that $f(\bar e)=-f(e)$ and the sum of all the values $g(v_1)+f(e_1)+g(v_2)+f(e_2)+...+g(v_n)+f(e_n)$ along any cycle of $G$ is zero. We say that this function $f:\mathbb E\rightarrow A$ balances the function $g:V\rightarrow A$.
%\ede
%\bde The set of all the balanceable functions $g:V\rightarrow A$ is denoted by $\mathcal{BF}(V,A)$. The group $\mathcal{BF}(V,A)$ is a subgroup of the free Abelian group $A^V$ of all the functions from $V$ to $A$.
%\ede
\bde A function $h:V\bigcup\mathbb E\rightarrow G$, which takes both vertices and edges of $\Gamma$ to some elements of $G$, is called balanced if $h(\bar e)=(h(e))^{-1}$ and the product of its values $h(v_1)\cdot h(e_1)\cdot h(v_2)\cdot h(e_2)\cdots h(v_n)\cdot h(e_n)$ along any cycle of $\Gamma$ is $1$.
\ede
\bde The set of all the balanced functions $h:V\bigcup\mathbb E\rightarrow G$ is denoted by $\mathcal{FU}(\Gamma,G)$. %The group $\mathcal{WF}(G,A)$ is a subgroup of the Abelian group $A^{V\bigcup\mathbb E}$ of all the functions from $V\bigcup\mathbb E$ to $A$.
\ede
%Clearly, any balanced function $f\in\mathcal{FE}(\mathbb E,G)$ can be viewed as a balanced function from $V\bigcup\mathbb E$ to $A$ which takes zero value on every vertex of $G$. Thus, we will regard $\mathcal{HF}(\mathbb E,A)$ as a subgroup of $\mathcal{WF}(V\bigcup\mathbb E,A)$.
%\bpr The quotient $\mathcal{WF}(V\bigcup\mathbb E,A)/\mathcal{HF}(\mathbb E,A)$ is naturally isomorphic to $\mathcal{BF}(V,A)$.
%\epr
%\begin{proof} The natural isomorphism is defined by ``forgetting" the values of $h\in \mathcal{WF}(V\bigcup\mathbb E,A)$ on the edges of $G$ and regarding it just as a balanceable function from $V$ to $A$.
%\end{proof}
%We review some basic definitions and facts regarding groups.
%\bde The order $ord(a)$ of an element $a\in G$ is the minimal positive number such that $a^{ord(a)}=1$. If no such positive number exists we say that $ord(a)=\infty$.
%\ede
\bde The set of all involutions of $G$ is denoted $G_2$. I.e., $G_2=\left\{a\in G\,|\,a^2=1\right\}$. %The set $A_2$ is a subgroup of $A$.
\ede
%\bde The image of the doubling map from $A$ to itself, which multiplies every element of $A$ by $2$, is a subgroup of $A$ denoted by $2A$.
%\ede
%Notice that $A_2$ is the kernel of the doubling map.
%\\ \\
The set $\mathcal{FE}(\mathbb E,G)$ is well understood and the following fact is well known.
\bpr\label{HFL}
For a finite group $G$, the cardinality of the set $\mathcal{FE}(\mathbb E,G)$ is equal to $|G|^{|V|-1}$.
\epr
\begin{proof}
Select a vertex $v$ and notice that the number of all functions $g:V\rightarrow G$ such that $g(v)=1$ is equal to $|G|^{|V|-1}$.  Consider the following bijection between the set of all $G$-valued functions $g$ on $V$ with $g(v)=1$ and the set $\mathcal{FE}(\mathbb E,G)$. For any such $g$, since each edge $e\in \mathbb E$ goes from some vertex $x$ to some vertex $y$, we define $f(e)=(g(x))^{-1}\cdot g(y)$. A straightforward calculation shows that $f\in\mathcal{FE}(\mathbb E,G)$. In the other direction of the bijection, for $f\in\mathcal{FE}(\mathbb E,G)$ we inductively construct the function $g$ as follows: we set $g(v)=1$; if $g(u)$ has been defined for a vertex $u$ then for every vertex $w$, for which there exists some edge $e$ from $u$ to $w$, we define $g(w)=g(u)\cdot f(e)$. Since $f\in\mathcal{FE}(\mathbb E,G)$, any two calculations of the value of $g$ on any vertex $u$ will produce the same result. The weak connectivity implies that every vertex indeed receives a value, and thus, our $g$ is well-defined.
\end{proof}
Our main result is the following.
\begin{thm}
Let $\Gamma=(V,E)$ be a weakly connected directed graph and $\Gamma'$ be its underlying undirected graph. Then:
\begin{enumerate}
\item If $\Gamma'$ is bipartite, then $|\mathcal{FU}(V\bigcup\mathbb E,G)|=|G|^{|V|}$.
\item If $\Gamma'$ is not bipartite, then $|\mathcal{FU}(V\bigcup\mathbb E,G)|=|G_2|\cdot |G|^{|V|-1}$.
\end{enumerate}
\end{thm}
\begin{proof}
If $\Gamma$ consists only of one vertex then part (1) of our theorem is trivial. Otherwise, let us look at any one non-loop edge of $\Gamma$ as it is depicted in Fig.~\ref{OneEdge}:
\begin{figure}[htp]
 $$\xygraph{
!{<0cm,0cm>;<1cm,0cm>:<0cm,1cm>::}
!{(0,1) }*+{a}="1"
!{(3,1) }*+{b}="2"
"1":"2"^p  }$$
\caption{An edge with values on it and on its origin and end.}\label{OneEdge}
\end{figure}

The letters on the edge and the vertices denote the values of a function $h:V\bigcup\mathbb E\rightarrow G$. Assume that $h$ is balanced, i.e., $h\in\mathcal{FU}(V\bigcup\mathbb E,G)$. Then for the cycle obtained by walking along this edge and returning back along it we have the equation $apbp^{-1}=1$, which immediately implies that $apb=p$.

\noindent
(1) Let $\Gamma'$ be bipartite. Then $\Gamma$ does not have cycles of odd length.

Let us construct a bijection between the
 sets $\mathcal{FU}(V\bigcup\mathbb E,G)$ and $\{(a,f)\,|\,a\in G\,,\,f\in\mathcal{FE}(\mathbb E,G)\}$. Such a bijection will imply $|\mathcal{FU}(V\bigcup\mathbb E,G)|=|\{(a,f)\,|\,a\in G\,,\,f\in\mathcal{FE}(\mathbb E,G)\}|=|G|\cdot |G|^{|V|-1}=|G|^{|V|}$. Let $v_1, e_1,v_2,e_2 \dots, v_{2k},e_{2k}$ be a cycle in $\Gamma$. Take $h\in \mathcal{FU}(V\bigcup\mathbb E,G)$. Then $h(v_1)\cdot h(e_1)\cdot h(v_2)\cdot h(e_2) \cdots h(v_{2k})\cdot h(e_{2k})=1$. As we noticed above, $h(v_j)\cdot h(e_j)\cdot h(v_{j+1})=h(e_j)$, therefore, $h(e_1)\cdot h(e_2) \cdots  h(e_{2k})=1$. So, define $f\in\mathcal{FE}(\mathbb E,G)$ as a restriction of $h$ on $\mathbb E$, i.e., $f(e)=h(e)$ for every $e\in\mathbb E$. Now fix a certain vertex $v$, and correspond to our $h\in \mathcal{FU}(V\bigcup\mathbb E,G)$ the pair $(h(v),f)$. Obviously, this correspondence is a bijection between the sets $\mathcal{FU}(V\bigcup\mathbb E,G)$ and $\{(a,f)\,|\,a\in G\,,\,f\in\mathcal{FE}(\mathbb E,G)\}$. Indeed, take a pair $(a,f)$, where $a\in G$ and $f\in\mathcal{FE}(\mathbb E,G)$ and define $h(e)=f(e)$ for every $e\in\mathbb E$ and $h(v)=a$. Let $e\in E$ be an edge which goes from the chosen vertex $v$ to a vertex $w$, put $h(w)=h(e)^{-1}a^{-1}h(e)$. Since $\Gamma'$ is connected, we can continue this way, and every vertex will receive a value, which depends only on values of $f$ on edges and on the element $a\in G$ that we started with.

\noindent
(2) Now assume that $\Gamma'$  is not bipartite. Then $\Gamma$ has a cycle of odd length.

Let us construct a bijection between the sets $\mathcal{FU}(V\bigcup\mathbb E,G)$ and $\{(a,f)\,|\,a\in G_2\,,\,f\in\mathcal{FE}(\mathbb E,G)\}$. Such a bijection will imply $|\mathcal{FU}(V\bigcup\mathbb E,G)|=|\{(a,f)\,|\,a\in G\,,\,f\in\mathcal{FE}(\mathbb E,G)\}|=|G_2|\cdot |G|^{|V|-1}$. Let $v_1, e_1,v_2,e_2 \dots, v_{2k+1},e_{2k+1}$ be a cycle of odd length in $\Gamma$. Take $h\in \mathcal{FU}(V\bigcup\mathbb E,G)$. Then $h(v_1)\cdot h(e_1)\cdot h(v_2)\cdot h(e_2) \cdots h(v_{2k+1})\cdot h(e_{2k+1})=1$. Again, as we noticed above, $h(v_j)\cdot h(e_j)\cdot h(v_{j+1})=h(e_j)$, therefore, $h(e_1)\cdot h(e_2) \cdots  h(e_{2k})\cdot h(v_{2k+1})\cdot h(e_{2k+1})=1$. Since $h(v_{2k+1})\cdot h(e_{2k+1})\cdot h(v_{1})=h(e_{2k+1})$, we get $h(e_1)\cdot h(e_2) \cdots  h(e_{2k})\cdot h(e_{2k+1})=h(v_1)$. In the same way walking along this cycle in the opposite direction we get $(h(e_{2k+1}))^{-1}\cdot (h(e_{2k}))^{-1} \cdots  (h(e_2))^{-1}\cdot (h(e_1))^{-1}=h(v_1)$. Thus, $h(v_1)=(h(v_1))^{-1}$, i.e., $h(v_1)\in G_2$. To obtain from $h\in \mathcal{FU}(V\bigcup\mathbb E,G)$ a balanced function $f$ on edges we do the following: define $f(e)=h(v)h(e)$ where the vertex $v$ is the origin of the edge $e$. Since $h$ is balanced on $V\bigcup\mathbb E$ it is easy to see that $f$ on $\mathbb E$. So, to $h\in \mathcal{FU}(V\bigcup\mathbb E,G)$ we correspond the pair $(h(v_1),f)$, where $h(v_1)\in G_2$ and $f\in\mathcal{FE}(\mathbb E,G)$.

 Conversely, take a pair $(a,f)$ where $a\in G_2$ and $f\in\mathcal{FE}(\mathbb E,G)$, and construct $h\in \mathcal{FU}(V\bigcup\mathbb E,G)$ as follows. Choose a vertex $v$ put $h(v)=a$.  For any edge $e\in \mathbb E$ which starts at the vertex $v$ put $h(e)=h(v)f(e)$. Let $w$ be the end of an edge $e\in \mathbb E$ whose origin is $v$. Put $h(w)=(h(e))^{-1}h(v)h(e)=(f(e))^{-1}af(e)$. Since $\Gamma'$ is connected, we can continue this way, and thus we define values of $h$ on all vertices and edges. By direct calculation it can be easily seen that $h$ is balanced. Let us look at the example depicted in Fig.~\ref{triangle} which illustrates our argument.

\begin{figure}[htp]
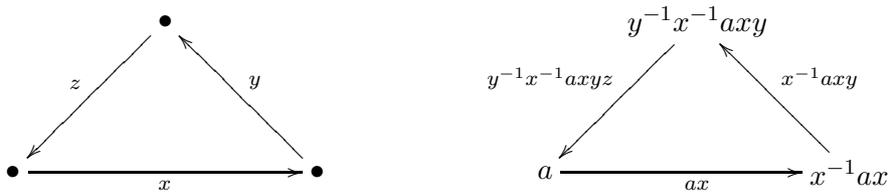
\label{triangle}
\begin{center}
%\[\begin{tikzpicture}[x=1.3cm, y=1cm,
 %   every edge/.style={
%        draw,
 %       postaction={decorate,
  %                  decoration={markings,mark=at position 0.5 with {\arrow{>}}}
  %                 }
   %     }
%]
	%\vertex (ul) at (0,1) {};
%	\vertex[fill] (u) at (1,1) {};
%	\vertex[fill] (ll) at (0,0) {};
%	\vertex[fill] (lr) at (2,0) {};
%	\path
%		(lr) edge[thick] (u)
%		(ll) edge[thick] (lr)
%		(ll) edge[thick] (u)
		%(ll) edge[bend right=20] (ur)
		%(ur) edge[bend right=20] (ll)
%	;
%\end{tikzpicture}\]
$$\xygraph{
!{<0cm,0cm>;<1cm,0cm>:<0cm,1cm>::}
!{(0,1) }*+{\bullet}="1"
!{(4,1) }*+{\bullet}="2"
!{(2,3) }*+{\bullet}="3"
!{(7,1) }*+{a}="4"
!{(11,1) }*+{x^{-1}ax}="5"
!{(9,3) }*+{y^{-1}x^{-1}axy}="6"
"1":"2"_{x} "2":"3"_{y} "3":"1"_{z}
"4":"5"_{ax} "5":"6"_{x^{-1}axy} "6":"4"_{y^{-1}x^{-1}axyz}}
$$
\caption{Constructing $h$ from $(a,f)$.}
\end{center}
\end{figure}
In this picture $x$, $y$, $z$ are values of a balanced function on edges, i.e., $xyz=1$, and $a\in G_2$, i.e., $a^2=1$. The right triangle represents a balanced function on vertices and edges as described above:
$$
a\cdot ax\cdot x^{-1}ax\cdot x^{-1}axy\cdot y^{-1}x^{-1}axy\cdot y^{-1}x^{-1}axyz=1\,\,,
$$
$$
a\cdot ax\cdot x^{-1}ax\cdot(ax)^{-1}=1\,\,,
$$
$$
x^{-1}ax\cdot x^{-1}axy\cdot y^{-1}x^{-1}axy\cdot(x^{-1}axy)^{-1}=1\,\,,
$$
$$
y^{-1}x^{-1}axy\cdot y^{-1}x^{-1}axyz\cdot a\cdot (y^{-1}x^{-1}axyz)^{-1}=1\,\,.
$$
It is easy to see that two constructed mappings, $h\mapsto(a,f)$ and $(a,f)\mapsto h$ are inverse to each other.
\end{proof}

\section{The rigid case}
Let $\Gamma=(V,E)$ be a weakly connected directed graph. Recall that in the rigid case we are allowed to walk only in the direction of an edge but not against it. It naturally changes the notion of a path and of a cycle in comparison with the flexible case.
\bde\label{rpathdef} A path $P$ from a vertex $x$ to a vertex $y$ is an alternating sequence $v_1$, $e_1$, $v_2$, $e_2$,...,$v_{n}$, $e_{n}$ of vertices from $V$ and different edges from $E$ (and not $\mathbb E$) such that $v_1=x$ and each $e_j$, for $j=1,...,n-1$, goes from $v_j$ to $v_{j+1}$ and $e_n$ goes from $v_n$ to $y$.
\ede
Obviously, not every flexible cycle is a rigid cycle.
%For example, the triangle depicted in Fig.~\ref{triangle2} is a cycle in the flexible case but is not a cycle in the rigid case.
%\begin{figure}[htp]
%\begin{center}
%\[\begin{tikzpicture}[x=1.3cm, y=1cm,
 %   every edge/.style={
 %       draw,
  %      postaction={decorate,
  %                  decoration={markings,mark=at position 0.5 with {\arrow{>}}}
  %                 }
   %     }
%]
	%\vertex (ul) at (0,1) {};
%	\vertex[fill] (u) at (1,1) {};
%	\vertex[fill] (ll) at (0,0) {};
%	\vertex[fill] (lr) at (2,0) {};
%	\path
%		(lr) edge[thick] (u)
%		(ll) edge[thick] (lr)
%		(ll) edge[thick] (u)
		%(ll) edge[bend right=20] (ur)
		%(ur) edge[bend right=20] (ll)
%	;
%\end{tikzpicture}\]
%$$\xygraph{
%!{<0cm,0cm>;<1cm,0cm>:<0cm,1cm>::}
%!{(0,1) }*+{\bullet}="1"
%!{(4,1) }*+{\bullet}="2"
%!{(2,3) }*+{\bullet}="3"
%"1":"2" "2":"3" "1":"3"}$$
%\caption{A flexible cycle which is not a rigid cycle.}\label{triangle2}
%\end{center}
%\end{figure}

Similarly to the flexible case denote by $\mathcal{RE}(E,G)$ and $\mathcal{RU}(V\bigcup E,G)$ the sets of balanced $G$-valued functions on edges and balanced functions of the entire graph $\Gamma$ (vertices and edges), respectively.
%\bpr
%Any function on  the set of vertices is balanceable. I.e., $\mathcal{BR}(V,A)=A^V$.
%\epr
%\begin{proof} Let  us take a function on vertices $g:V\rightarrow A$. Define the function $h:V\bigcup E\rightarrow A$ in the following way: $h(v)=g(v)$ for any vertex $v\in V$, $h(e)=-g(v)$ for all the edges $e\in E$ which start at $v$. Obviously $h$ is a balanced function.
%\end{proof}
\bde Two vertices $x$ and $y$ of $\Gamma$ are called strongly connected if exists a path $P_1$ from $x$ to $y$ and a path $P_2$ from $y$ to $x$. We also say that every vertex is strongly connected to itself.
\ede
Note: We did not require that $P_1$ and $P_2$ do not have common edges.
\bde\label{rcycledef} A cycle is a path $P$ from a vertex $x$ to itself. %We denote the set of all the cycles of $G$ by $C(G)$.
\ede
Note: Since the above-mentioned $P_1$ and $P_2$ paths might have common edges, $P_1$ followed by $P_2$ might not be a cycle. There could even not exist any cycle, containing both $x$ and $y$. Consider the following example:
\begin{exa} Consider the graph $\Gamma$ depicted in Fig.~\ref{StrongCon} with $V(\Gamma)=\{x,v,w,y\}$ and $E(\Gamma)=\{e_1,e_2,e_3,e_4,e_5\}$.% where $e_1=(x,v)$, $e_2=(y,v)$, $e_3=(w,x)$, $e_4=(w,y)$ and $e_5=(v,w)$.% as depicted in Figure~\ref{F1}.
 The path $P_1=x,e_1,v,e_5,w,e_4$ is the only path which goes from $x$ to $y$ and the path $P_2=y,e_2,v,e_5,w,e_3$ is the only path which goes from $y$ to $x$. They have a common edge $e_5$. Thus, according to Definitions~\ref{rpathdef} and~\ref{rcycledef}, there exists no cycle, containing both $x$ and $y$.
\begin{figure}[htp]
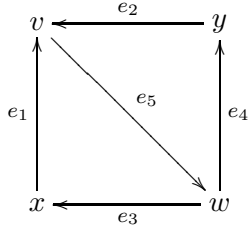

$$\xygraph{
!{<0cm,0cm>;<1.2cm,0cm>:<0cm,1.2cm>::}
!{(1,1) }*+{x}="x"
!{(1,3) }*+{v}="v"
!{(3,3) }*+{y}="y"
!{(3,1)}*+{w}="w"
"x":"v"^{e_1}   "v":"w"^{e_5}
"w":"y"_{e_4}   "w":"x"^{e_3}  "y":"v"_{e_2}
}$$
\caption{The vertices $x$, $y$ are strongly connected but no cycle contains both of them.}\label{StrongCon}
\end{figure}
\end{exa}
Strong connectivity defines an equivalence relation on the vertices of $\Gamma$. The equivalence classes of strongly connected vertices, together with all the edges between the vertices of each class, are called the strongly connected components of $\Gamma$. We denote the number of strongly connected components of $\Gamma$ by $\bar{k}(\Gamma)$.
Recall that $G^S$ usually denotes the set of all $G$-valued functions on the set $S$.
\begin{lem}\label{HR} If $\bar{k}(\Gamma)=1$ then $|\mathcal{RE}(E,G)|=|G|^{|V|-1}$ just like in the flexible case.
\end{lem}
\begin{proof} Let $E=\{e_1,...,e_n\}$. Edge $e_1$ goes from some $x$ to some $y$. There is a path $P$ which goes from $y$ to $x$ and does not contain $e_1$, since if $P$ contains $e_1$ we can just delete this $e_1$ and all the vertices and edges, which come after it, from $P$. Thus, the product of values of any $f\in\mathcal{RE}(E,G)$ along $P$ must be equal to $(f(e_1))^{-1}$. Hence we can define the graph $\Gamma_1$ adding a new edge $\bar{e}_1$ to $\Gamma$ which goes from $y$ to $x$ and we can extend the function $f$ to a balanced function on edges of the new graph $\Gamma_1$ setting $f(\bar{e}_1)=(f(e_1))^{-1}$. So there is a natural bijection between the set of the balanced functions on edges of $\Gamma_1$ after the addition of $\bar{e}_1$  and the set of the balanced functions on edges of original graph $\Gamma$ before the addition of the edge $\bar{e}_1$. Repeating this process for all the edges of $E$ we reduce the problem to the flexible case, while not changing the set of the balanced functions on edges of $\Gamma$.
\end{proof}
\bpr\label{directed} Let $\bar{k}(\Gamma)$ be the number of strongly connected components of $\Gamma$, and $r(\Gamma)$ be the number of all the edges in $\Gamma$ which go from a vertex in one strongly connected component of $\Gamma$ to a vertex in a different strongly connected component of $\Gamma$.\\ Then $|\mathcal{RE}(E,G)|=|G|^{|V|-\bar{k}(\Gamma)+r(\Gamma)}$.
\epr
\begin{proof} Let $V_1,...,V_t$ be the equivalence classes of vertices $\Gamma$ with respect to the relation of strong connectivity, and denote the set of edges between the vertices of $V_j$ by $E_j$. Then, obviously, $\mathcal{RE}(E,G)=\mathcal{RE}(E_1,G)\times\cdots\times \mathcal{RE}(E_t,G)\times G^{U}$, where $U$ is the set of all the edges between the vertices in different strongly connected components of $\Gamma$. By Lemma~\ref{HR} we conclude that $|\mathcal{RE}(E,G)|=|G|^{|V_1|-1+|V_2|-1+\cdots+|V_t|-1+r(\Gamma)}=|G|^{|V|-\bar{k}(\Gamma)+r(\Gamma)}$.
\end{proof}
\bpr
%Use the notations as in Proposition~\ref{directed}.
%Let $\bar{k}(\Gamma)$ be the number of strongly connected components of $\Gamma$, and $r(\Gamma)$ be the number of all the edges in $\Gamma$ which go from a vertex in one strongly connected component of $\Gamma$ to a vertex in a different strongly connected component of $\Gamma$.
Notations as above.
Then $|\mathcal{RU}(V\bigcup E,G)|=|G|^{2|V|-\bar{k}(\Gamma)+r(\Gamma)}$.
\epr
\begin{proof}
To every $h\in\mathcal{RU}(V\bigcup E,G)$ corresponds the pair $(g,f)$, where $g\in G^V$ is just the restriction of $h$ on vertices, and the value of $f\in \mathcal{RE}(E,G)$ on every edge $e$ is equal to $h(v)h(e)$, where the vertex $v$ is the origin of the edge $e$. Such a function $f$ is obviously a balanced function on edges since its value along any path is equal to the value of $h$ along that path. This correspondence between the elements of $\mathcal{RU}(V\bigcup E,G)$ and the pairs from $G^V\times \mathcal{RE}(E,G)$ is a bijection. Indeed, for a given pair $(g,f)$, where $g$ is any function on vertices and $f$ is a balanced function on edges, we can construct $h:V\bigcup E\rightarrow G$ as follows: $h(v)=g(v)$ for all $v\in V$ and $h(e)=(g(v))^{-1}f(e)$ for all $e\in E$, where the vertex $v$ is the origin of the edge $e$. The constructed bijection implies that $|\mathcal{RU}(V\bigcup E,G)|=|G|^{|V|}\cdot |\mathcal{RE}(E,G)|=|G|^{2|V|-\bar{k}(\Gamma)+r(\Gamma)}$.
\end{proof}
Thus, the flexible problem for a graph $\Gamma=(V, E)$ can be regarded as the rigid problem for the graph $\Gamma'=(V, \mathbb E)$. Vice versa, the rigid problem for a graph $\Gamma$ can be regarded as a free product of the rigid problems for the strongly connected components of $\Gamma$ multiplied by $G^{r(\Gamma)}$ where $r(\Gamma)$ is the number of edges between different strongly connected components of $\Gamma$.

%The following simple claim connects this work to~\cite{CGL}.
%\bpr Let $G$ be an undirected connected graph and let $G_{dir}$ be a directed graph obtained from $G$ by any assignment of directions to the edges of $G$. Denote by $H(E,A)$ the group of $A$-valued balanced functions on edges of $G$. Choose any order on edges of $G$ and embed $H(E,A)$ and $\mathcal{HR}(E(G_{dir}),A)$ into $A^{|E|}$. For an undirected graph $G$ the group of balanced functions on edges of $G$ is equal to the intersection of all the groups $\mathcal{HR}(E(G_{dir}),A)$, where $G_{dir}$ runs over all directed graphs for all $2^{|E|}$ possible direction assignments to the edges of $G$. The same is true for the groups of balanced functions on the entire graph (both vertices and edges). Namely, $W(V\bigcup E,A)=\bigcap\mathcal{WR}(V\bigcup E(G_{dir}),A)$.
%\epr
%\begin{proof}
%Let $Cyc=v_1,e_1,...,v_k,e_k$ be a cycle in the undirected graph $G$. There exists a directed graph $G_{dir}$ for which $c$ is also a cycle. So any $f\in \bigcap\mathcal{HR}(E(G_{dir}),A)$ must satisfy the equation $\sum_{i=1}^kf\left(e_i\right)=0$. Therefore $f\in H(E,A)$, since $Cyc$ is an arbitrary cycle of $G$. Hence,
%$$H(E,A)\supseteq\bigcap\mathcal{HR}(E(G_{dir}),A)\,.$$ The opposite inclusion is obvious, since any cycle of any $G_{dir}$ is a cycle of $G$. The proof of the second statement of the proposition is the same.
%\end{proof}

\section{Conclusions and further research}
Notice that although the sets of non-Abelian-group-valued balanced function on edges and balanced functions on vertices and edges do not inherit any natural group structure, their enumeration for directed graphs is similar to the Abelian case studied in~\ref{CGL2}. For undirected graphs the non-Abelian case is much more involved, almost nothing is known for today.

A natural way to extend the findings of the present paper is to consider balanced functions on a graph $\Gamma=(V,E)$ with values in a finite semigroup $X$. Since a semigroup does not necessarily have an identity element, the requirement of vanishing on cycles can be replaced by the following: choose an element $a\in X$ and require that the products of values of a function $f$ along any cycle is equal to $a$. When the semigroup $X$ is commutative, if instead of fixing $a\in X$ we require that the products along all the cycles are equal to each other, then the set of these functions also form a semigroup, and it is natural to inquire what is the structure of this semigroup of functions. The last question is not so obvious even for Abelian groups. When the semigroup $X$ is non-commutative the question of enumeration arises. Also there is another way to generalize this problem for semigroups:
for each directed edge $e$ which goes from a vertex $v$ to a vertex $w$ require $f(v)f(e)f(w)=f(e)$, and for any cycle $v_1$, $e_1$, $v_2$, $e_2$,...,$v_n$, $e_n$ require $f(v_1)f(e_1)f(v_2)f(e_2)\cdots f(v_n)f(e_n)f(v_1)=f(v_1)$.

% When the group is non-Abelian, the set of balanced functions does not inherit any natural group structure, however, it is still an interesting object. It is rather obvious that the number of balanced functions on edges in the flexible case is $|X|^{|V|-1}$, since the arguments of the proof of Proposition~\ref{HFL} can be repeated with appropriate slight changes. However, if we allow our functions to take values on both vertices and edges, the simple observation that the values on the ends of an edge must be inverse one to another looks now as $apbp^{-1}=1$ (see Fig.~\ref{OneEdge}), which means that the values on two vertices connected by an edge must be in relation $a^{-1}=pbp^{-1}$, i.e., $a^{-1}$ must be conjugate to $b$. An attempt to extend our results in this direction likely may lead to some results on probabilities on finite groups in the spirit of~\cite{DN10},~\cite{ND},~\cite{DN12},~\cite{CGLS}.


\begin{thebibliography}{1}
\bibitem{BHN} Roland Bacher, Pierre de la Harpe, Tatiana Nagnibeda. The lattice of integral flows and the lattice of integral cuts on a finite graph. Bulletin del la S.M.F. {\bf 125}, $167-198$, (1997).
\bibitem{OI} R. Balakrishnan and N. Sudharsanam. Cycle vanishing edge valuations of a graph. Indian J. Pure Appl. Math. {\bf 13} (3), $313-316$, (1982).
\bibitem{BeckEtAl} M. Beck, F. Breuer, L. Godkin, J. L. Martin. Enumerating colorings, tensions and flows in cell complexes. Journal of Combinatorial Theory, Series A {\bf 122}, $82-106$, (2014).
\bibitem{BreuerDall} F. Breuer, A. Dall. Bounds on the Coefficients of Tension and Flow Polynomials. Journal of Algebraic Combinatorics {\bf 33}, $465-482$, (2011).
\bibitem{BCNS} V. Bugaenko, Y. Cherniavsky, T. Nagnibeda, R. Shwartz. Weighted Coxeter graphs and generalized geometric representations of Coxeter Groups. Discrete Applied Mathematics, to appear. Preprint available from http://arxiv.org/abs/1304.0692.
\bibitem{Chen} Beifang Chen. Orientations, Lattice Polytopes, and Group Arrangements I: Chromatic and Tension Polynomials of Graphs. Annals of Combinatorics {\bf 13}, $425-452$, (2010).
%\bibitem{Biggs} Norman Biggs, Algebraic Graph Theory, Cambridge Mathematical Library (2nd ed.), Cambridge University Press, (1994) p. 53, ISBN 9780521458979.
\bibitem{CGL} Y. Cherniavsky, A. Goldstein and V. E. Levit. Groups of balanced labelings on graphs. Discrete Mathematics {\bf 320}, $15-25$, (2014). %Preprint available from http://arxiv.org/abs/1301.4206. Preprint available at http:$//$arxiv.org$/$abs$/1301.4206$.
\bibitem{CGL2} Y. Cherniavsky, A. Goldstein and V. E. Levit. Balanced Abelian group valued functions on directed graphs. Preprint available at http:$//$arxiv.org$/$abs$/1303.5456$.
%\bibitem{CGLS} Y. Cherniavsky, A. Goldstein, V. E. Levit, R. Shwartz. Probabilities of permutation equalities in finite groups. Preprint available from http://arxiv.org/abs/1403.3868.
%\bibitem{DN10} A.K. Das, R.K. Nath. On generalized relative commutativity degree of a finite
%group. Int. Electron. J. Algebra {\bf 7}, $140-151$, (2010).
%\bibitem{DN12} A.K. Das, R.K. Nath. A generalization of commutativity degree of finite groups.
%Comm. Algebra {\bf 40}, $1974-1981$, (2012).
\bibitem{Diestel} R. Diestel. Graph Theory. Springer; 4th edition 2010. Corrected 2nd printing 2012 edition.
%\bibitem{BJG} J{\o}rgen Bang-Jensen, Gregory Gutin. Digraphs: Theory, Algorithms and Applications, 2nd Edition. Springer-Verlag, London
%Springer Monographs in Mathematics. ISBN: 978-1-84800-997-4
%\bibitem{BL} Itai Benjamini, Laszlo Lovasz. Harmonic and analytic functions on graphs. Journal of Geometry, 76 (2003)
%\bibitem {Bol} B. Bollobas. Modern Graph Theory. Volume 184 of Graduate Texts in Mathematics, Springer, Berlin, 1998.
%\bibitem{Deo} N. Deo. Graph Theory with Applications to Engineering and Computer Science. Prentice-Hall Series in Automatic Computation.
%Prentice-Hall, Englewood Cliffs, 1982.
%\bibitem{Harary} Harary, F. Graph Theory. Reading, MA: Addison-Wesley, 1994
\bibitem{NI} Manas Joglekar, Nisarg Shah, Ajit A. Diwan. Balanced group-labeled graphs. Discrete Mathematics {\bf 312}, $1542-1549$, (2012).
\bibitem{K}  Martin Kochol. Tension polynomials of graphs. Journal of Graph Theory {\bf 40}, $137-146$, (2002).
%\bibitem {HKM}  Ramesh Hariharan , Telikepalli Kavitha , Kurt Mehlhorn. A faster deterministic algorithm for minimum cycle basis in directed graphs. In proceedings of ICALP (2006).
%\bibitem{Harju} Tero Harju. Lecture Notes on Graph Theory. Department of Mathematics, University of Turku, Finland.
%\bibitem{Menger} Menger, Karl. "Zur allgemeinen Kurventheorie". Fund. Math. 10: 96–115 (1927)
\bibitem{KY} Martin Kreissig, Bin Yang. Efficient Synthesis of Consistent Graphs. EURASIP, 2010, $1364-1368$.
%\bibitem{LR} C. Liebchen and R. Rizzi. A Greedy Approach to compute a Minimum Cycle Basis of a Directed Graph. Information
%Processing Letters, 94(3): 107-112, (2005).
%\bibitem{ND} R.K. Nath, A.K. Das. On generalized commutativity degree of a finite group. Rocky Mountain J. Math {\bf 41}, $1987-2000$, (2011).
\bibitem{Serre} Jean-Pierre Serre. Arbres, amalgames, $SL_2$ (1977; Trees).
%\bibitem{Tsin} Yung H. Tsin. A Simple 3-Edge-Connected Component Algorithm.  Theory of Computing Systems, Volume 40, Number 2 (2007), pages 125-142.
\bibitem{T} W. T. Tutte. A Contribution to the Theory of Chromatic Polynomials. Canadian Journal of Mathematics {\bf 6}, $80-91$, (1954).
\bibitem{Woodall} Douglas R. Woodall. Tutte polynomial expansions for 2-separable graphs. Discrete Mathematics {\bf 247}, $201-213$, (2002).
\bibitem{Z}
T. \ Zaslavsky. A Mathematical Bibliography of
Signed and Gain Graphs and Allied Areas, The Electronic Journal of Combinatorics, Dynamic Surveys, $\#$DS8 (2012).
\end{thebibliography}
\end{document}